\documentclass[11pt]{amsart}
\usepackage[english]{babel}
\usepackage{latexsym}
\usepackage{indentfirst}
\usepackage{amsxtra}
\usepackage{amssymb}
\usepackage{url}
\usepackage{amsmath}
 \usepackage{mathrsfs}
\usepackage{hyperref}

 \usepackage[T1]{fontenc}
\usepackage[utf8]{inputenc}
\usepackage[all]{xy}
\usepackage{tikz}

 \usepackage{graphicx}
 \usepackage{tikz-cd}
\usepackage{comment}
\usepackage{cancel}
\usepackage{mathtools}
\usepackage{faktor}
\usepackage{xfrac}    
\graphicspath{ {./images/} }
\usepackage{caption}
\usepackage{xcolor}

\newcommand{\Z}{\mathbb{Z}}

\newcommand{\id}[1]{\Id(#1)}

\newtheorem{lemma}{Lemma}[section]
\newtheorem{teo}[lemma]{Theorem}
\newtheorem{thm}[lemma]{Theorem}
\newtheorem*{thm*}{Theorem}
\newtheorem{prop}[lemma]{Proposition}
\newtheorem{oq}[lemma]{Open question}
\newtheorem{coro}[lemma]{Corollary}
\newtheorem{defi}[lemma]{Definition}
\newtheorem{rem}[lemma]{Remark}
\newtheorem{ex}[lemma]{Example}
\newtheorem{nex}[lemma]{Non-example}

\newcommand{\St}{\operatorname{St}_{2n}}

 \DeclareMathOperator{\Id}{Id}

\DeclareMathOperator{\sgn}{sgn}

\def\gp#1{\langle#1\rangle}
\newcommand{\fx}{F\gp{X}}

\begin{document}
\title[PI-Quivers]{Quivers with Polynomial Identities}

\author[G.~Cerulli~Irelli]{Giovanni Cerulli Irelli}
\address{Dipartimento SBAI, Sapienza Universit\`a di Roma, Via Scarpa 10, 00161, Roma, Italy}
\email{giovanni.cerulliirelli@uniroma1.it}

\author[J.~De~Loera~Chávez]{Javier De Loera Chávez}
\address{Dipartimento SBAI, Sapienza Universit\`a di Roma, Via Scarpa 16, 00161, Roma, Italy}
\email{javieralejandro.deloerachavez@uniroma1.it}

\author[E.~Pascucci]{Elena Pascucci}
\address{Dipartimento SBAI, Sapienza Universit\`a di Roma, Via Scarpa 16, 00161, Roma, Italy}
\address{CMUP, Departamento de Matemática, Faculdade de Ciências, Universidade do Porto, Rua do Campo Alegre 687, 4169-007, Porto, Portugal}
\email{elena.pascucci@uniroma1.it}
\email{elena.pascucci@fc.up.pt}

\thanks{The authors were supported by “Progetti
di Ateneo” of Sapienza Universit\`a di Roma, by GNSAGA-INDAM, and by NextGenerationEU - PRIN 2022 -B53D23009430006 - 2022S97PMY -PE1-investimento M4.C2.1.1-Structures for Quivers, Algebras and Representations (SQUARE)}



\begin{abstract} 
We provide a topological characterization of quivers whose path algebra satisfies a polynomial identity. This class includes the oriented cycle and  acyclic quivers and,  in the latter case, we describe the associated T-ideal. We introduce a generalization of Arnold's A-graded algebras, which we call locally A-graded algebras, and prove that they are also PI.
We give an example of a quiver algebra satisfying a polynomial identity, even if the path algebra of the quiver does not.
\end{abstract}

\maketitle

\section{{Introduction}}

An  algebra over a field $F$ is called PI if there exists a non-commutative polynomial whose evaluation in every tuple of its elements is zero. Such a polynomial is then called a polynomial identity for the algebra. For example, if $A$ is commutative then a polynomial identity  for $A$ is $x_{1}x_{2}-x_{2}x_{1}$.  We say that a quiver $Q$ is PI if its path algebra $FQ$ is $PI$. Our first result gives a classification of finite PI-quivers.

\begin{thm}\label{Thm:PIQuiverInto}
A finite quiver $Q$  is $PI$ if and only if every vertex of $Q$ lies in at most one oriented cycle.
\end{thm}
Since subalgebras and quotients of PI algebras are PI, Theorem~\ref{Thm:PIQuiverInto} provides a large new class of PI algebras. 

The quivers appearing in Theorem~\ref{Thm:PIQuiverInto} are not new in the literature: Pere Ara informed us that  they are precisely the quivers for which the corresponding Leavitt path algebra has finite GK-dimension \cite[Theorem~1.1]{Pere}, \cite{AAJZelmanov}. The class of Leavitt path algebras associated to a finite quiver which are PI are precisely those of GK dimension less or equal than one  \cite[Theorem~4.1]{BLR}. 

In the light of Theorem~\ref{Thm:PIQuiverInto}, an acyclic quiver $Q$ is PI. In this case we can improve considerably the result and give a description of the T-ideal $\id{FQ}$ of polynomial identities for its path algebra, at least in case the field $F$ has characteristic zero.

\begin{thm}\label{Thm:AcyclicIntro}
Let $F$ be a field of characteristic zero and let $Q$ be an acyclic quiver. Then 
\(
\id{FQ}=\langle [x_1,x_2]\cdots [x_{2m-1},x_{2m}] \rangle_T
\)
where $m-1$ is the maximal length of a path in $Q$.
\end{thm} 
In the process of proving Theorem~\ref{Thm:PIQuiverInto} we encountered a new class of  algebras that we call \emph{locally A-graded} (definition~\ref{Def:LocAGraded}), which are a generalization of Arnold's A-graded algebras \cite{Arnold}. Examples of locally A-graded algebras include matrix algebras, upper triangular matrix algebras and their nilpotent subalgebras, the commutative square quiver algebra, the path algebra of an oriented cycle, the Nakayama algebras. Our second result is as follows
\begin{thm}\label{Thm:LocAGradedInto}
A locally A-graded algebra is PI. 
\end{thm}

By Theorem~\ref{Thm:PIQuiverInto} we know that the quiver $Q$ obtained by gluing two oriented cycles in a vertex $v$ is not PI. However, the last result of this paper states that a natural quotient of $Q$ is indeed PI.
\begin{thm}\label{Thm:TwoCyclesIntro}
Let $Q$ be the quiver obtained by gluing two oriented cycles $C(1)$ and $C(2)$ on a vertex $v$. Let $\alpha$ be an arrow of $C(1)$ ending in $v$  and let $\beta$ be an arrow of $C(2)$ starting from $v$, then the algebra $A=FQ/(\alpha \beta)$ is PI. 
\end{thm}

 The paper is organized as follows: in Section~\ref{Sec:PI} we recall facts about PI theory. In Section~\ref{Sec:Quivers} we recall basic notions on quivers and quiver algebras. In Section~\ref{Sec:Unicursal} we introduce our main tool to prove Theorem~\ref{Thm:LocAGradedInto}, namely the graph theoretic machinery used by  Swan to prove Amitsur-Levitski's Theorem; this section contains also Lemma~\ref{Lemma:DisconnectingArrows} which is used in the proof of Theorem~\ref{Thm:PIQuiverInto}. We prove Theorems~\ref{Thm:LocAGradedInto}, \ref{Thm:PIQuiverInto}, \ref{Thm:AcyclicIntro}, \ref{Thm:TwoCyclesIntro} in Sections~\ref{Sec:LocAgraded}, \ref{Sec:PIquivers}, \ref{Sec:Acyclic}, \ref{Sec:PIQuiverRelations}, respectively.
\newline\newline
\noindent
\textbf{Acknowledgements}. This project has been supported by the PRIN 2022 SQUARE and it has been presented in the related conference SQUARE, partially supported by CIRM, which took place at the CIRM facility in Vason in June 2025. We thank the organizers and the participants of that conference for many stimulating discussions. We thank Pere Ara for pointing out the connection with Leavitt path algebras and for his careful reading of a previous version of this paper. EP thanks Pere Ara for the kind hospitality in Barcelona. We thank Azzurra Ciliberti for useful discussions. We thank Kulumani Rangaswamy to suggest the open question~\ref{OQ:Rangaswamy}. We thank the anonymous referee for the useful comments and suggestions.

\section{Polynomial Identities}\label{Sec:PI}

In this section, we report basic notions about the theory of polynomial identities. Standard references are \cite{AGPRmono} and \cite{GZb}. We fix a field $F$. By an algebra we mean an associative algebra over $F$. Given a countable set of non-commutative variables $X=(x_{1},x_{2},\dots)$ we denote by $F\langle X\rangle=F\langle x_{1},x_{2},\dots\rangle$ the free associative algebra generated by $X$. Its elements are hence non-commutative polynomials in $X$ with coefficients in $F$. 
\begin{defi}
Let $A$ be an $F$-algebra and $f \in F\langle x_{1},\dots, x_{n}\rangle$.
We say that $f$ is a \emph{polynomial identity for $A$} (or simply a \emph{PI for $A$}) if $f(a_1, \dots, a_n)=0$ for all $a_1, \dots, a_n \in A$. If there exists a non-zero polynomial identity for $A$ then $A$ is called \emph{PI}.
\end{defi}

\begin{ex}\label{Ex:CommutativeAlgebra}
A commutative algebra is PI. Indeed, it satisfies the polynomial 
$[x_1, x_2]\coloneqq x_1x_2- x_2x_1,$ called the \emph{Lie commutator} of $x_1$ and $x_2$. 
\end{ex}
\begin{ex}\label{Ex:NilpotentAlgebra}
A nilpotent algebra $A$ is $PI$. Indeed, the polynomial $x_1^n$ is a polynomial identity for $A$, where $n$ is its nilpotency index. 
\end{ex}

\begin{lemma}\label{Lem:2x2Matrices}
The algebra $M_2(F)$ of $2\times 2$ matrices over $F$ is PI. Indeed, the polynomial in three variables $[x_1, [x_2,x_3]^2]$ is a PI for it. 
\end{lemma}
\begin{proof}
By the Cayley-Hamilton theorem, $A^{2}-\mathrm{tr}(A)A+\mathrm{det}(A)\mathbf{1}_{2}=0$, for every $A\in M_2(F)$. In particular, if $A=[B,C]$ is a commutator then $A^{2}=-\mathrm{det}(A)\mathbf{1}_{2}$ is a central element. Thus $[A,[B,C]^{2}]=0$  for every $A,B,C\in M_{2}(F)$.
\end{proof}
\begin{ex}\label{Ex:FiniteDimAlgebra}
A finite dimensional algebra is $PI$. Indeed,  the determinant in $n+1$ variables is a polynomial identity, where $n$ denotes its dimension.
\end{ex}
\begin{nex}\label{Nex:FreeAlgebra}
The free algebra $F\langle x_{1},\dots, x_{n}\rangle$ in $n\ge 2$ variables does not satisfy any polynomial identity. Indeed, the existence of such an identity would contradict its freeness.
\end{nex}
\begin{defi}
Given an $F$-algebra $A$, we denote by
\[
\id{A}\coloneqq \{f \in \fx \ | \ \text{f is PI for A} \}
\]
the \emph{set of polynomial identities of $A$}.
\end{defi}
We recall that $\id A$ is a \emph{$T$-ideal}, which means that it is a two-sided ideal of $\fx$ stable under every endomorphism of $\fx$. In simple terms, this means that
given a polynomial \( f(x_1, \dots, x_n) \in \Id(A) \), if each variable \( x_i \) is replaced by an arbitrary polynomial \( g_i \in F\langle X \rangle \) then  
\(f(g_1, \dots, g_n) \in \id{A}\).
The following useful properties of T-ideals hold 
\begin{align} \label{Eq:Subalgebra}
&\textrm{If }B\subseteq A\textrm{ is a subalgebra, }\id{A}\subseteq \id{B},\\\label{Eq:IdQuotAlgebra}
&\textrm{If }B=A/I\textrm{ is a quotient algebra, then }\id{A}\subseteq \id{B}.
\end{align}
In particular, subalgebras and quotients of PI algebras are PI algebras.

Given  non-empty subset $Y \subseteq F\langle X\rangle$, we denote by $\langle Y  \rangle_T$ the $T$-ideal generated by $Y$ in $F\langle X\rangle$.

The following remarkable result provides a  description of the T-ideal of the algebra $UT_n(F)$ of upper triangular $n\times n$ matrices with entries in $F$.

\begin{thm}\cite[Theorem~1]{Malcev}\label{Thm:Malcev}
Let $F$ be a field of characteristic zero. Then 
\(
\id{UT_n(F)}= \langle u_n\rangle_T
\)
where 
\begin{equation}\label{Eq:un}
u_n(x_1,\dots, x_{2n})\coloneqq [x_1,x_2][x_3,x_4]\dots [x_{2n-1},x_{2n}].
\end{equation}
\end{thm}

The hypothesis on the characteristic of the field $F$ in Theorem~\ref{Thm:Malcev} is mainly due to the following result, which holds only in characteristic zero. Recall that a polynomial $f$ is said to be \emph{multilinear} if it is linear in every variable, i.e. its degree with respect to each variable is one.

\begin{teo}\label{char 0 ordinary}
 Let $F$ a field of characteristic zero and let $A$ be an $F$-algebra. Then $\id A$ is generated by multilinear polynomials. 
 \end{teo}

 An important example of multilinear polynomial is the \emph{standard polynomial} of degree $n$:
 \begin{equation}\label{Def:Stn}
 St_{n}(x_1, \dots, x_n) \coloneqq \sum_{\sigma \in S_n} \text{sgn}({\sigma}) x_{\sigma(1)} \cdots x_{\sigma(n)},
 \end{equation}
 where $S_n$ denotes the \emph{symmetric group} of degree $n$ and $\text{sgn}({\sigma})$ is the sign of the permutation $\sigma$. Indeed, the following remarkable result holds about the algebra $M_n(F)$ of $n\times n$ matrices over $F$.
 
 \begin{teo}[Amitsur-Levitski]\cite{AmitsurLevitski1940}\label{Thm:Amitsur-Levitski} 
 For every $n\ge 1$, $St_{2n}\in\id{M_n(F)}$.
 \end{teo}
Notice that the following Leibniz type formula  holds
\[
 St_{n+1}(x_1, \dots, x_{n+1})=\sum_{i=1}^{n+1}(-1)^{i+1}\,x_i\, St_{n}(x_1,\dots, \hat{x_i},\dots, x_{n+1})
 \]
 so that 
 \begin{equation}\label{Eq:LeibnizStn}
 St_m\in \langle St_n\rangle_T, \;\forall \,m\ge n.
 \end{equation}
 The standard polynomial is an example of an alternating polynomial.
 Recall that a polynomial $f\coloneqq f(x_1, \dots,x_n; y_1,\dots, y_m)$, which is multilinear in the first $n$ variables, is \emph{alternating} in $x_1, \dots, x_n$ if, for every $i \in \{1, \dots, n\}$,
 \[ f(x_1, \dots, x_i, \dots, x_i, \dots, x_n; y_1, \dots, y_m)=0_{F\langle X\rangle}.\]

\section{Quivers and Path Algebras}\label{Sec:Quivers}

In this section we fix notation and recall some basic facts about quivers and path algebras. 
Standard references are \cite{CB}, \cite{ARS}, \cite{Ringel}, \cite{ASS}  and \cite{Schiffler}. 

\smallskip

A \emph{quiver} $Q$ is a quadruple $Q=(Q_{0},Q_{1},s,t)$ where
\begin{itemize}
    \item $Q_{0}$ is a finite set of \emph{vertices},
    \item $Q_{1}$ is a finite set of \emph{arrows},
    \item $s, t: Q_1 \to Q_0$ are two functions assigning to each arrow $\alpha \in Q_1$ its \emph{starting} vertex $s(\alpha)$ and its \emph{terminal} vertex $t(\alpha)$.
\end{itemize}

An arrow $\alpha\in Q_{1}$ is graphically represented as:
\begin{tikzcd}
   s(\alpha) \arrow[r, "\alpha"] & t(\alpha)
\end{tikzcd}

\begin{ex}
An example of quiver is given by:
\begin{center}
\begin{tikzcd}
    1 \arrow[r, "\alpha"] & 2 \arrow[r, "\beta"] & 3
\end{tikzcd}
\end{center}
where $Q_0 = \{1,2,3\}$, $Q_1 = \{\alpha, \beta\}$, $1=s(\alpha)$, $2=t(\alpha)=s(\beta)$, $3=t(\beta)$.
\end{ex}

A \emph{path} $p$ of length $n=l(p)\ge1$ in a quiver $Q$ is a sequence $p=(\alpha_{1},\alpha_{2},\dots,\alpha_{n})$ of $n$ arrows of $Q$ such that  $t(\alpha_i) = s(\alpha_{i+1})$ for all $1 \leq i \leq n-1$; the vertex $s(p)\coloneqq s(\alpha_1)$ is the \emph{starting vertex} of $p$ and the vertex $t(p)\coloneqq t(\alpha_n)$ is the \emph{terminal vertex} of $p$; the \emph{support} of $p$ is the subquiver of $Q$ whose arrow set is   
$\{\alpha_{1},\dots,\alpha_{n}\}$ (thus every arrow of $p$ appears precisely once in the support of $p$) and whose vertices are the starting and terminal vertex of each of its arrows. Instead of $p=(\alpha_{1},\alpha_{2},\dots, \alpha_{n})$ we use the shorthand notation $p=\alpha_{1}\alpha_{2}\dots\alpha_{n}$.  To every vertex $v\in Q_{0}$ it is formally attached the path $e_{v}$ which has length zero and starts and ends at $v$. The path $e_{v}$ is called the \emph{lazy path} at $v$. A \emph{closed path} in $Q$ is a path $\pi$ of positive length such that $t(\pi) = s(\pi)$.  An \emph{oriented cycle} of $Q$ is the support of a closed path. 
An \emph{$n$-cycle}  is a closed path with $n$ vertices and $n$ arrows and it will be denoted by $C_n$. A \emph{loop} is a $1$-cycle. A quiver is said to be \emph{acyclic} if it contains no oriented cycles. 
\begin{ex}
Let $Q=\begin{tikzpicture}[>=latex, baseline=(1.base)]
\node (1) at (0,0) {$1$};
\draw[->] (1) to[out=160, in=200, loop] node[left] {$\alpha$} (1);
\draw[->] (1) to[out=20, in=-20, loop] node[right] {$\beta$} (1);
\end{tikzpicture}
$ be the quiver with two arrows and one vertex. Then $\pi=\alpha\beta\alpha\beta\alpha$ is a closed path and its support is $Q$. The paths $\alpha$ and $\beta$ are  loops and the quivers $\begin{tikzpicture}[>=latex, baseline=(1.base)]
\node (1) at (0,0) {$1$};
\draw[->] (1) to[out=160, in=200, loop] node[left] {$\alpha$} (1);
\end{tikzpicture}$, $\begin{tikzpicture}[>=latex, baseline=(1.base)]
\node (1) at (0,0) {$1$};
\draw[->] (1) to[out=20, in=-20, loop] node[right] {$\beta$} (1);
\end{tikzpicture}$
and $Q$ are oriented cycles. 
\end{ex}
Given a quiver $Q$ and a field $F$, the \emph{path algebra} $FQ$ of $Q$ is the $F$-vector space spanned by all possible paths in $Q$, including lazy paths.
The multiplication of two paths $p$ and $q$ in $Q$ is the path $pq$ if $t(p)=s(q)$ and it is zero otherwise.  The multiplication in $FQ$ is defined by extending linearly the  path multiplication. The set of all paths in $Q$ is called the \emph{standard basis} of $FQ$. A path algebra $FQ$ is finite dimensional if and only if $Q$ is acyclic.

The path algebra $FQ$ is graded by the length of the paths, i.e. there is a direct sum decomposition $FQ=\bigoplus_{d\ge0} FQ^{d}$ where $FQ^{d}$ is the span of all paths in $Q$ of length $d$. Moreover, the set of lazy paths $\{e_{i}\mid\,i\in Q_{0}\}$ is a complete set of primitive orthogonal idempotents of $FQ$. Thus $FQ$ admits a decomposition $FQ=\bigoplus_{i,j\in Q_{0}}e_{i}FQe_{j}$. Here, $e_{i}FQe_{j}$ is  hence the vector subspace of $FQ$ generated by paths starting at vertex $i$ and ending at vertex $j$. Combining these two decompositions we get 
\begin{equation}\label{Eq:FQLocAGraded}
FQ=\bigoplus_{d,i,j} FQ^{d}_{i,j}
\end{equation}
where $FQ^{d}_{i,j}\coloneqq FQ^{d}\cap e_{i}FQe_{j}$.

Let $\mathfrak{m}$ denote the two-sided ideal of $FQ$ generated by $Q_{1}$. A two sided ideal $I\subset FQ$ is called \emph{admissible} if there exists a positive integer $k$ such that $\mathfrak{m}^{k}\subseteq I\subseteq \mathfrak{m}^{2}$. The quotient of a path algebra by an admissible ideal is called a \emph{bound quiver algebra}. In particular, a bound quiver algebra is finite dimensional. A quotient $FQ/I$ where $I\subseteq \mathfrak{m}^{2}$ is simply called a \emph{quiver algebra}.

If an admissible ideal $I$ of $FQ$ is generated by paths in $Q$ then $I$ is called a \emph{monomial} ideal and  $FQ/I$ is called a \emph{monomial quiver algebra}. It has a basis formed by paths called its \emph{standard basis}. 
\begin{ex}\label{Ex:PathAlgebra}
For the quiver
\begin{tikzcd}
    1 \arrow[r, "\alpha"] & 2 \arrow[r, "\beta"] & 3
\end{tikzcd}
the path algebra $FQ$ has standard basis $\{e_1, e_2, e_3, \alpha, \beta, \alpha\beta\}$. The ideal $I=\langle\alpha\beta\rangle$ is admissible and the bound quiver algebra $FQ/I$ has standard basis $\mathcal{B}=\{e_1, e_2, e_3, \alpha, \beta\}$.
\end{ex}

\begin{ex}\label{Ex:Loop}
   If \(Q=\begin{tikzpicture}[>=latex, baseline=(1.base)]
\node (1) at (0,0) {$1$};
\draw[->] (1) to[out=160, in=200, loop] node[left] {$\alpha$} (1);
\end{tikzpicture}\) then $FQ\simeq F[x]$, the polynomial ring in one variable.
\end{ex}

\begin{ex}
If $Q=\begin{tikzpicture}[>=latex, baseline=(1.base)]
\node (1) at (0,0) {$1$};
\draw[->] (1) to[out=160, in=200, loop] node[left] {$\alpha$} (1);
\draw[->] (1) to[out=20, in=-20, loop] node[right] {$\beta$} (1);
\end{tikzpicture}$ then $FQ\simeq F\langle x_1,x_2\rangle$.
\end{ex}
\begin{lemma}\label{Lem:AnUpperTriangular}
Let $Q=\overrightarrow{A\,}_{n}$ be the equioriented quiver of type $A_{n}$, i.e.
\begin{equation}\label{An}
\overrightarrow{A\,}_n\coloneqq 
\begin{tikzcd}
1 \arrow[r, "\alpha_{1}"] & 2 \arrow[r, "\alpha_{2}"] & \cdots \arrow[r, "\alpha_{n-1}"] & n.
\end{tikzcd}
\end{equation}
Then  $FQ\cong UT_n(F)$.
\end{lemma}
\begin{proof}
The map $p_{i,j}\mapsto e_{ij}$, where $p_{ii}=e_i$, $p_{i,j}=\alpha_{i}\cdots\alpha_{j-1}$ and $e_{ij}$ is the matrix whose only non-zero entry is $1$ at $(i,j)$, is an isomorphism.
\end{proof}

\section{Unicursal paths}\label{Sec:Unicursal}
For convenience of the reader, in this section we collect some results and notions from \cite{Swan} and also prove Lemma~\ref{Lemma:DisconnectingArrows}, to be used later on. 

\begin{defi}\cite[Definition~2]{Swan}
    Let $\Sigma=(\Sigma_{0},\Sigma_{1},s,t)$ be a finite quiver with arrow set $\Sigma_{1} = \{\alpha_1, \ldots, \alpha_m\}$. A path $\gamma$ in $\Sigma$ is called \emph{unicursal} if it contains all arrows of $\Sigma$ exactly once. For a unicursal path $\gamma=\alpha_{\sigma(1)}\cdots \alpha_{\sigma(m)}$ we define $\operatorname{sgn}(\gamma) \coloneqq  \operatorname{sgn}(\sigma)$.
\end{defi}
Unicursal paths are most commonly known as \emph{Eulerian tours} \cite[Ch.~10]{Stanley}.

The \emph{flow} of a vertex $v$ of a quiver $\Sigma$ is the number of arrows starting from $v$ minus the number of arrows ending in $v$. 
\begin{ex}
The quiver in Example~\ref{Ex:PathAlgebra} has one vertex of flow $1$, one vertex of flow $-1$ and one vertex of flow zero.
\end{ex}
\begin{prop}\cite[Proposition~1]{Swan}\label{Prop:Swan}
Let $\Sigma$ be a quiver without isolated vertices and let $i,j\in \Sigma_0$ be two of its vertices. If there is a unicursal path from $i$ to $j$ then $\Sigma$ is connected, every vertex other than $i$ and $j$ has flow zero, if $i=j$ then $i$ has flow zero, if $i\neq j$ then the flow of $i$ is $1$ and the flow of $j$ is $-1$.
 \end{prop}
 The following result was used by Swan to give another proof of the Amitsur-Levitski Theorem~\ref{Thm:Amitsur-Levitski}. 
 \begin{thm}\cite[Theorem~2]{Swan}\label{gra1}
    Let $G$ be a quiver such that 
    \begin{equation}\label{Eq:SwanQuiver}
    \lvert G_{1}\rvert \geq 2\lvert G_{0}\rvert.
    \end{equation}
For any $i,j\in G_{0}$ the number of unicursal paths from $i$ to $j$ with sign $1$ is the same as the number of unicursal paths from $i$ to $j$ with sign $-1$. 
\end{thm}
We notice that if a quiver $\Sigma$ admits a unicursal path then it is \emph{path-connected} in the following sense: for every two distinct vertices $i,j\in\Sigma_0$ either there is a path from $i$ to $j$ or from $j$ to $i$. Moreover, if every vertex of $\Sigma$ has flow zero, then $\Sigma$ is \emph{strongly path-connected}, i.e. for every two vertices $i,j\in\Sigma_0$ there is a path from $i$ to $j$ and a path from $j$ to $i$. 

We will need the following lemma.
\begin{lemma}\label{Lemma:DisconnectingArrows}
Let $\Sigma$ be a  quiver admitting a unicursal path. If there exists an arrow $\alpha\in \Sigma_1$ which does not lie in an oriented cycle, then $\Sigma\setminus \{\alpha\}$ is disconnected.
\end{lemma}
\begin{proof}
Let $\pi=\gamma_1\alpha\gamma_2$ be a unicursal path in $\Sigma$. Let $\Sigma(1)$ and $\Sigma(2)$ denote the support of $\gamma_{1}$ and  $\gamma_{2}$, respectively. We claim that $\Sigma\setminus\{\alpha\}=\Sigma(1)\sqcup \Sigma(2)$.  
Suppose, for a contradiction, that  there exists a vertex $k$ in common between $\Sigma(1)$ and $\Sigma(2)$. 
Since $s(\alpha)$ is a vertex of $\Sigma(1)$, there exists a path $p_{1}$ in $\Sigma(1)$ from $k$ to $s(\alpha)$. Since $t(\alpha)$ is a vertex of $\Sigma(2)$, there exists a path $p_2$ in $\Sigma(2)$ from $t(\alpha)$ to $k$. It follows that $p_1 \alpha p_2$ is an oriented cycle containing $\alpha$, against the hypothesis.
\end{proof}

We will use several times the following definition, implicitly used in \cite{Swan}.
\begin{defi}\label{Def:SwanQuiver}
Let $A$ be an algebra which,  as a vector space, admits a decomposition of the form $A=\bigoplus_{i,j=1}^nA_{i,j}$ such that $A_{ij}A_{jk}\subseteq A_{i,\ell}$ and $A_{ij}A_{k\ell}$ is zero if $j\neq k$.  Let $\beta=(p_1,\dots p_m)\subset A$ be a finite subset of  elements of $A$ such that $p_1p_2\cdots p_m\neq0$ and for every $h$ there exists two indices $i_h$ and $j_h$ such that $p_h\in A_{i_hj_h}$. The \emph{Swan quiver} of $\beta$ is the quiver $\Sigma_\beta$ whose set of vertices is $\{i_h, j_h\mid h=1,\dots, m\}$ and whose arrows are $p_h: i_h\rightarrow j_h$.
\end{defi}
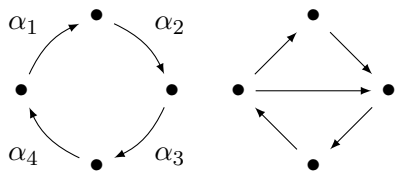
\begin{figure}
\centering
\begin{tabular}{cc}
\begin{tikzpicture}[>=latex]
\node (1) at (0,0) {$\bullet$};
\node (2) at (1,1) {$\bullet$};
\node (3) at (2,0) {$\bullet$};
\node (4) at (1,-1) {$\bullet$};
\draw[->] (1) to[bend left=20] node[above left] {$\alpha_1$} (2);
\draw[->] (2) to[bend left=20] node[above right] {$\alpha_2$} (3);
\draw[->] (3) to[bend left=20] node[below right] {$\alpha_3$} (4);
\draw[->] (4) to[bend left=20] node[below left] {$\alpha_4$} (1);
\end{tikzpicture}
&
\begin{tikzpicture}[>=latex]
\node (1) at (0,0) {$\bullet$};
\node (2) at (1,1) {$\bullet$};
\node (3) at (2,0) {$\bullet$};
\node (4) at (1,-1) {$\bullet$};
\draw[->] (1)--(2); 
\draw[->] (2) --(3);
\draw[->] (3) --(4);
\draw[->] (4) --(1);
\draw[->] (1) --(3);
\end{tikzpicture}
\end{tabular}
\caption{An example of a quiver and of a Swan quiver}\label{Fig:QuiverSwan}
\end{figure}

\begin{ex}
Let us consider the quiver $Q$ shown on the left of Figure~\ref{Fig:QuiverSwan}. Let $\beta=(\alpha_1, \alpha_2, \alpha_3, \alpha_4, \alpha_1\alpha_2)$. The Swan quiver $\Sigma_\beta$ is shown on the right of Figure \ref{Fig:QuiverSwan}.
\end{ex}

\section{Locally A-graded algebras}\label{Sec:LocAgraded}
In this section we introduce the notion of \emph{locally A-graded algebras} and prove that they admit a polynomial identity.

\begin{defi}\label{Def:LocAGraded}
Given a positive integer $n$, a \emph{locally A-graded} algebra of size $n$, is a $\Z$-graded $F$-algebra $A=\bigoplus_{d\in\Z }A^{d}$ with the property that each homogeneous vector subspace is a direct sum $A^d=\bigoplus_{1\le i,j\le n}A_{i,j}^d$ such that

\begin{align}\label{Def:AGraded1}
&\textrm{dim}(A_{i,j}^d)\le 1,
\\\label{Def:AGraded2}
&\textrm{there exists an homogenous linear basis }\mathcal{B}=\{e_{ij}^{(d)}\}\textrm{ of }A \\\notag &\textrm{which is matrix-like, i.e. }e_{ij}^{(d)}\in A_{i,j}^d, \textrm{ and}
\end{align}
\[
        e_{ij}^{(d_1)} e_{k\ell}^{(d_2)}=\begin{cases}
        \delta_{j,k}  e_{i,\ell}^{(d_1+d_2)} &\textrm{if } A_{i,\ell}^{d_1+d_2}\not =0,\\
        0 &\text{otherwise},
         \end{cases}
 \]
 where $\delta_{jk}$ is the Kronecker delta. We say that $\mathcal{B}$ is the \emph{standard basis} of $A$.
\end{defi}
\begin{rem}
If $A$ is a  locally A-graded algebra for which there exist $d$, $i$ and $j$ such that $A_{i,j}^t\neq 0$ for all $0\le t\le d$ then $\bigoplus_{t=0}^dA_{i,j}^t$ is $A$-graded in the sense  of Arnold \cite{Arnold}. This explains the choice of terminology. 
\end{rem}
\begin{ex}\label{Ex:LocAgr}
Let $Q$ be the quiver 
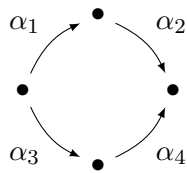
\[
\begin{tikzpicture}[>=latex]
\node (1) at (0,0) {$\bullet$};
\node (2) at (1,1) {$\bullet$};
\node (3) at (2,0) {$\bullet$};
\node (4) at (1,-1) {$\bullet$};
\draw[->] (1) to[bend left=20] node[above left] {$\alpha_1$} (2);
\draw[->] (2) to[bend left=20] node[above right] {$\alpha_2$} (3);
\draw[->] (1) to[bend right=20] node[below left] {$\alpha_3$} (4);
\draw[->] (4) to[bend right=20] node[below right] {$\alpha_4$} (3);
\end{tikzpicture}
\]
The bound quiver algebra $A=FQ/\langle\alpha_1\alpha_2-\alpha_3\alpha_4\rangle$ (also called the commutative square algebra)
is locally A-graded. 
\end{ex}
\begin{nex}
The path algebra $FQ$ of the quiver in the previous example is not locally A-graded (with respect to the standard basis) since there are two different paths with the same endpoints and of the same length.
\end{nex}
\begin{ex}
The algebra $M_n(F)$  is locally A-graded with $\mathcal{B}=\{e_{ij}\}$.
\end{ex}

\begin{ex}
The path algebra  of a quiver $T$ whose underlying graph is a tree is locally A-graded. Indeed, given two vertices there is at most one path between them. In this case, the stronger condition $\mathrm{dim}\,e_{i}FT e_{j}\le 1$ holds.
\end{ex}
\begin{lemma}\label{Lem:Cn}
The path algebra of $C_{n}$ is locally A-graded. 
\end{lemma}
\begin{proof}
Given two vertices $i$ and $j$ of $C_n$, there is at most one path of a given length $d$ starting in $i$ and ending in $j$. 
\end{proof}
\begin{nex}\label{Nex:PathAlgNotLocAgr}
The path algebra of the  quiver 
\begin{equation}\label{Eq:QuiverTwoLoops}
\begin{tikzpicture}[>=latex]
\node (1) at (0,0) {$\bullet$};
\node (2) at (1,0) {$\bullet$};
\draw[->] (1) to[out=160, in=200, loop] node[left] {$\alpha_1$} (1);
\draw[->] (1) to node[above] {$\alpha_2$} (2);
\draw[->] (2) to[out=20, in=-20, loop] node[right] {$\alpha_3$} (2);
\end{tikzpicture}
\end{equation}
is not locally A-graded. Indeed, the two distinct paths $\alpha_1\alpha_1\alpha_2\alpha_3$ and $\alpha_1\alpha_2\alpha_3\alpha_3$ have the same endpoints and length.
\end{nex}
\begin{oq}
It would be interesting to characterize quivers whose path algebra is locally A-graded. 
\end{oq}
The next lemma provides more examples of locally A-graded algebras.
\begin{lemma}\label{Lem:SubalgebraLocAGraded}
Let $A$ be a locally A-graded algebra with standard basis $\mathcal{B}$. If $B\subset A$ is a subalgebra of $A$ spanned by a part of $\mathcal{B}$ then $B$ is locally A-graded. Dually, if $I\subset A$ is a two-sided ideal spanned by a part of $\mathcal{B}$ then $A/I$ is locally A-graded. 
\end{lemma}
\begin{proof}
It follows immediately from the definition.
\end{proof}
\begin{coro}\label{Cor:QuotientPathLocGraded}
Let $Q$ be a quiver whose path algebra $A=FQ$ is locally A-graded. Then, for every monomial ideal $I\subset FQ$, the quiver algebra $A/I$ is locally A-graded.
\end{coro}
\begin{proof}
By hypothesis, $I$ is generated by paths in $Q$. The result is hence a consequence of Lemma~\ref{Lem:SubalgebraLocAGraded}.
\end{proof}
\begin{ex}
Given $n\ge 2$, the algebra $UT_n(F)$ of upper triangular $n\times n$ matrices is locally A-graded. In particular, the path algebra of $\overrightarrow{A\,}_n$ is locally A-graded.
\end{ex}
\begin{ex}
The algebra of  nilpotent upper triangular matrices is locally A-graded.
\end{ex}
\begin{prop}
Every basic Nakayama algebra is locally A-graded. 
\end{prop}
\begin{proof}
Let $A$ be a basic Nakayama algebra. By \cite[Theorem~V.3.2]{ASS} there is an isomorphism $A\cong FQ/I$ where $Q$ is either $\overrightarrow{A\,}_n$ or $C_n$ for some $n$ and $I$ is an admissible ideal (in particular $A$ is finite dimensional).
It is known  that $I$ is a monomial ideal and hence we conclude by corollary~\ref{Cor:QuotientPathLocGraded}. 
\end{proof}
The following is our first main result.
\begin{thm}\label{teo bis}
Let $A$ be a locally A-graded algebra of size $n$. Then $\operatorname{St}_{2n}$ is a polynomial identity for $A$.
\end{thm}
\begin{proof}
Let $A$ be a locally $A$-graded algebra of size $n$. Since $\operatorname{St_{2n}}$ is multilinear and alternating, we need only to check evaluations at pairwise different elements of the multiplicative basis $\mathcal B\coloneqq \{e_{ij}^{(d)}\}$ of $A$. Let $\beta=\{\beta_1,\ldots, \beta_{2n}\}\subset \mathcal{B}$ be such a choice.  The Swan quiver $\Sigma_\beta$ of $\beta$ (see definition~\ref{Def:SwanQuiver}) satisfies \eqref{Eq:SwanQuiver}: Indeed, $\lvert(\Sigma_\beta)_0\rvert\le n$ and $\lvert(\Sigma_\beta)_1\rvert=\lvert\beta\rvert=2n$.

For every $\sigma\in S_{2n}$, the element $\gamma=\beta_{\sigma(1)}\beta_{\sigma(2)}\cdots\beta_{\sigma(2n)}$ belongs to a graded piece $A_{i,j}^d$ where  $d=deg(\beta_1)+\ldots+deg(\beta_{2n})$ and $i$ and $j$ are suitable vertices of $\Sigma_{\beta}$. Such an element $\gamma$ corresponds to a unicursal path from $i$ to $j$ in $\Sigma_{\beta}$. 
By Theorem \ref{gra1}, there is some unicursal path $\gamma'$ 
in $\Sigma_{\beta}$ from $i$ to $j$ with $\sgn(\gamma')=-\sgn(\gamma)$. 
Since $A$ is locally A-graded, $\gamma=\gamma'$ as elements of $A$ and hence they cancel each other out in the expansion of $\St(\beta_1, \ldots, \beta_{2n})$.
\end{proof}
As a corollary we find new examples of PI-algebras.
\begin{coro}\label{Cor:FCn}
    The path algebra of the $n$-cycle $C_n$ satisfies $St_{2n}$. In particular, $FC_n$ is a PI-algebra. 
\end{coro}

\begin{rem}
The minimal degree of a polynomial identity on $F C_n$ is $2n$. Indeed, there are $2n-1$ elements that have a unique non-zero product. More precisely, the Swan quiver $\Sigma_\beta$ of the set $\beta=(C_n)_0\cup (C_n)_1\setminus \{\alpha_n\}$ of all primitive idempotents  and of all arrows except one has a unique unicursal path. 
\end{rem}

By definition, Nakayama algebras are finite dimensional and thus PI. The following immediate corollary of Theorem~\ref{teo bis} gives a polynomial identity whose degree is independent of the dimension of the algebra.  
\begin{coro}
A basic Nakayama algebra $KQ/I$ satisfies $St_{2\lvert Q_0\rvert}$. 
\end{coro}

\begin{ex}
The algebra $FC_n/I$ where $I=\langle\alpha\beta\rangle\neq 0$ can be shown to be isomorphic to a \emph{fundamental algebra} (see \cite{AljadeffGiambrunoPascucciSpinelli}).
\end{ex}

\section{PI-quivers}\label{Sec:PIquivers}
In this section, we classify the quivers whose path algebras are PI. For simplicity, we introduce a name for them:
\begin{defi}
A quiver is called a \emph{PI}-quiver if its path algebra is $PI$.
\end{defi}
The following is our second main result.
\begin{teo}\label{Thm:PIQuivers}
A quiver $Q$ is $PI$ if and only if each vertex of $Q$ lies in at most one oriented cycle.
\end{teo}
\begin{proof}
If a quiver $Q$ contains a vertex $v$ which belongs to two oriented cycles $c_1$ and $c_2$, then the subalgebra $B(c_1,c_2)$ of $FQ$ generated by $c_1$ and $c_2$ is isomorphic to the free algebra in two generators. It follows from \eqref{Eq:Subalgebra} that $\id{FQ}\subseteq \id{B(c_1,c_2)}=\id{F\langle x_1,x_2\rangle}=\{0\}$ and hence $FQ$ is not PI.

Let us prove the other implication. Thus, let $Q$ be a quiver such that every vertex of $Q$ belongs to at most one oriented cycle. Let $n=\lvert Q_0\rvert$ be number of vertices of $Q$. We prove that 
$St_{2n}\in \id{FQ}$. More precisely, we claim that for any set  $\beta=(p_1,p_2,\cdots, p_{2n})$ of $2n$ paths in $Q$, we have
\begin{equation}\label{Eq:ClaimPI}
St_{2n}(p_1,p_2,\dots, p_{2n})=0.
\end{equation}
If the elements of $\beta$ are not all distinct, then \eqref{Eq:ClaimPI} follows by the fact that $St_{2n}$ is an alternating polynomial. Let us suppose that $\lvert\beta\rvert=2n$. If every product of all the elements of $\beta$ is zero then \eqref{Eq:ClaimPI} holds by definition \eqref{Def:Stn}. We may assume, up to a reordering of $\beta$, that $p_1p_2\dots p_{2n}\neq0$. 
We notice that every non-zero summand of $S_{2n}(p_1,\cdots, p_{2n})$ corresponds to a unicursal path in the Swan quiver $\Sigma_\beta$. 
We notice that the number of vertices of $\Sigma_\beta$ is less or equal the number of vertices of $Q$, which is $n$ and the number of arrows of $\Sigma_\beta$ is $\lvert\beta\rvert=2n$; thus $\Sigma_\beta$ satisfies  \eqref{Eq:SwanQuiver} and we can apply Swan's Theorem~\ref{gra1}.
The following is the key lemma to conclude the proof. 

\begin{lemma}\label{Lem:KeyLemPiQuiver}
Let $Q$ be a quiver in which each vertex lies in at most one oriented cycle. Let $\beta=(p_1,\dots, p_m)$ be a set of paths in $Q$ such that $p_1p_2\cdots p_m\neq 0$. Then given two vertices $i,j\in (\Sigma_\beta)_0$ of the Swan quiver of $\beta$, all unicursal paths in $\Sigma_\beta$ from $i$ to $j$ represent the same path in $Q$. 
\end{lemma}
For convenience of the reader, we postpone the proof of Lemma~\ref{Lem:KeyLemPiQuiver}  below and finish the proof of Theorem~\ref{Thm:PIQuivers} as follows: By the key Lemma~\ref{Lem:KeyLemPiQuiver} all those paths represent the same path in $Q$ and thus they cancel each other out in $S_{2n}(p_1,\cdots, p_{2n})$. We conclude that \eqref{Eq:ClaimPI} holds.
\end{proof}

\begin{proof}[Proof of Lemma~\ref{Lem:KeyLemPiQuiver}]
We denote by $A_\beta$ the subalgebra of $FQ$ generated by $\beta$. We proceed by induction on $m=\lvert\beta\rvert\ge 1$. If $m=1$, then there is only one unicursal path in $\Sigma_\beta$ and hence there is nothing to prove. Let us assume that $m\ge 2$. There are two possibilities, 
\begin{align}\label{Condition1ProofKeyLemma}
&\textrm{every arrow of }\Sigma_\beta\textrm{ lies in an oriented cycle of }\Sigma_\beta,\\\label{Condition2ProofKeyLemma}
&\textrm{there exists an arrow }\alpha \textrm{ of }\Sigma_\beta\\\notag &\textrm{that does not lie in an oriented cycle of }\Sigma_\beta. 
\end{align}
If \eqref{Condition1ProofKeyLemma} holds, we claim that there exists an oriented cycle $C\subset Q$ such that $\beta\subset FC$. Indeed, we notice that $\Sigma_\beta$ is strongly path connected (to see this, we notice that there exists a path from $i$ to $j$ or from $j$ to $i$ because $\Sigma_\beta$ has a unicursal path and now we conclude by the fact that every arrow of $\Sigma_\beta$ is part of an oriented cycle). Now,  $p_1$ is part of an oriented cycle of $\Sigma_\beta$ that we call $C(1)$ and thus there exists an oriented cycle $C$ in $Q$ such that every arrow of $C(1)$ is a path of $C$, in particular $p_1\in FC$. If $C(1)=\Sigma_\beta$ we are done. Otherwise, let $p'\in \beta$ but not an arrow of $C(1)$. Then there exists a path $q_1$ from any vertex $v\in (C(1))_0$ to $s(p')$ and another path $q_2$ from $t(p')$ to $v$ and thus an oriented cycle $C(2)$ which is the support of the closed path $q_1p'q_2$. Since $v$ is both in $C(1)$ and in $C(2)$, $v$ is a vertex of $Q$ lying in two oriented cycles; since there are no vertices in $Q$ lying in two different oriented cycles, we conclude that $C(1)$ and $C(2)$ represent the same oriented cycle $C$ of $Q$ and thus $p'\in FC$. It follows that $\beta\subset FC$.  Since there are no vertices of $Q$ lying in two different oriented cycles, then $C=C_n$ is an $n$-cycle for some $n\ge 1$. Since, by Lemma~\ref{Lem:Cn},  $FC$ is locally A-graded, then, by Lemma~\ref{Lem:SubalgebraLocAGraded}, $A_\beta\subset FC$ is locally A-graded as well. In particular, since all unicursal paths of $\Sigma_\beta$ have the same length, by fixing a starting and an ending vertex, any unicursal path with those chosen endpoints represents a unique element of $FQ$.

If \eqref{Condition2ProofKeyLemma} holds, then by Lemma~\ref{Lemma:DisconnectingArrows}, there is a decomposition of $\Sigma_\beta\setminus\{\alpha\}=\Sigma(1)\sqcup \Sigma(2)$ into disjoint subquivers. In particular, every unicursal path $\pi$ in $\Sigma_\beta$ can be written in the form  $\gamma_1\alpha\gamma_2$ where $\gamma_1$ is unicursal in $\Sigma(1)$ and $\gamma_2$ is unicursal in $\Sigma(2)$. By induction, every unicursal path in $\Sigma(1)$ starting in $s(\gamma_1)$ and ending in $s(\alpha)$ represents the same element of $FQ$ and every unicursal path in $\Sigma(2)$ starting in $t(\alpha)$ and ending in $t(\gamma_2)$ represents the same element of $FQ$. This concludes the proof.
\end{proof}

\begin{ex}
The quiver shown in \eqref{Eq:QuiverTwoLoops} is PI, but its path algebra is not locally A-graded (see Non-example~\ref{Nex:PathAlgNotLocAgr}).
\end{ex}
\begin{nex}
The quiver shown on the right of Figure~\ref{Fig:QuiverSwan} is not PI, because it contains a vertex which lies in two different oriented cycles.
\end{nex}
\begin{rem}
If every vertex of $\Sigma_\beta$ has flow zero, then \eqref{Condition1ProofKeyLemma} holds. The quiver on the right of Figure~\ref{Fig:QuiverSwan} shows that the converse is not true.
\end{rem}

For infinite quivers the situation is different.
\begin{lemma}\label{Lem:Ainfinite}
Let $A_{+\infty}^{right}$ be the infinite quiver with vertex set $\mathbb{N}=(1,2,3,\cdots)$ and  arrow set $(i\rightarrow i+1\mid\, i\in\mathbb{N})$. Then $FA_{+\infty}^{right}$ is not PI.
\end{lemma}
\begin{proof}
The algebra $FA_{+\infty}^{right}$ is the direct limit $FA_{+\infty}^{right}=\cup_{n\geq 1} UT_n(F)$ and $\id{FA_{+\infty}^{right}}= \cap_{n\geq 1} \id{UT_n(F)}=\{0\}$ (see also \cite{Garcia Yasumura}).
\end{proof}
\begin{rem}\label{Rem:Infinite}
In view of Lemma~\ref{Lem:Ainfinite}, Theorem~\ref{Thm:PIQuivers} does not hold for infinite quivers.
\end{rem}
\begin{oq}\label{OQ:Rangaswamy}
Infinite quivers whose Leavitt path algebra has finite GK-dimension are classified in \cite[Theorem~6.4]{Rangaswamy}. In view of the analogy between finite PI quivers and Leavitt path algebras with finite GK-dimension mentioned in the introduction, it is natural to ask if the quivers appearing in loc.cit. are PI.
\end{oq}

\section{Acyclic quivers}\label{Sec:Acyclic}

In this section we find generators for the T-ideal of finite dimensional path algebras, i.e. path algebras of acyclic quivers. Notice that among the class of PI algebras the subclass for which is known a description of the generators of the T-ideal is quite small. For example, the set of generators of the T-ideal of $M_n(F)$ is not known unless  $n=2$ in which case it is described in \cite{Drensky1990}. The following is our third main result.
(See \eqref{Eq:un} for the definition of $u_n$.)
\begin{thm}\label{Thm:AcyclicQuivers}
Let $F$ be a field of characteristic zero and let $Q$ be an acyclic quiver. Then 
\begin{equation}\label{Eq:AcyclicTIdeal}
\id{FQ}=\langle u_m \rangle_T
\end{equation}
where $m-1$ is the maximal length of a path in $Q$.
\end{thm}
\begin{proof}
A path of maximal length in $Q$ gives an embedding of $\overrightarrow{A\,}_{m}$ in $Q$, because $Q$ is acyclic and hence the vertices of this path are all distinct. Hence $F\overrightarrow{A\,}_{m}\subseteq FQ$ and, by \eqref{Eq:Subalgebra}, $\id{FQ} \subseteq \id{F\overrightarrow{A\,}_{m}}$. By Lemma~\ref{Lem:AnUpperTriangular} and Theorem~\ref{Thm:Malcev}, we get $\id{FQ}\subseteq\langle u_m\rangle_T$.

To prove the other inclusion, we notice that, since $Q$ is acyclic, for any paths $\gamma_1$ and $\gamma_2$ in $Q$, the element $[\gamma_1, \gamma_2]$ is either zero or a path up to sign; if it is a non-zero path, it has length $l(\gamma_1)+l(\gamma_2)\ge 1$, because $[e_i, e_j]=0$ for any $i,j\in Q_0$. Thus, a non trivial evaluation of $u_m$ will produce a concatenation of $m$ paths each of length at least 1, which is zero by hypothesis. Therefore $u_m\in \id{FQ}$ and hence $\langle u_m\rangle_T\subseteq \id{FQ}$.
\end{proof}
\begin{rem}
Let $A=FQ/I$ be a finite dimensional monomial quiver algebra and let $m-1$ be the maximal length of an element of its standard basis. If  $F\overrightarrow{A\,}_m\subseteq A$ then the same argument as above proves that
\(\id{A}= \langle u_m\rangle_T.\)
The embedding $F\overrightarrow{A\,}_m\subseteq A$ is necessary to have this equality, as the next Non-example shows.
\end{rem}
\begin{nex}
Let us consider the monomial quiver algebra \(FQ/I\) where 
\(
Q=\begin{tikzpicture}[>=latex, baseline]
\node (1) at (0,0) {$\bullet$};
\node (2) at (1,0) {$\bullet$};
\draw[->] (1) to node[above] {$\alpha$} (2);
\draw[->] (2) to [out=20, in=-20, loop] node[right] {$\beta$} (2);
\end{tikzpicture}
\)
and \(
I=\langle \beta^2\rangle\). We notice that the longest standard basis element of $A$ is $\alpha\beta$ but $F\overrightarrow{A}_3\not\subset A$. In this case, it is still true that $u_3\in\id{A}$ but it is not of minimal degree. Indeed, $u_2\in \id{A}$ as well. Furthermore, $\langle u_3\rangle_T\subsetneq  \langle u_2\rangle= \id{A}$. 
\end{nex}

\section{A glimpse into quiver algebras}\label{Sec:PIQuiverRelations}

In this section, we take a first look at the behavior of quotients of non-PI quivers. The simplest examples of such quivers arise by gluing two oriented cycles at a common vertex.  We denote by $C_n\star_v C_m$ the quiver obtained by gluing $C_n$ and $C_m$ at a vertex $v$ (see figure~\ref{Fig:CnCm} for the case $n=3$ and $m=4$)
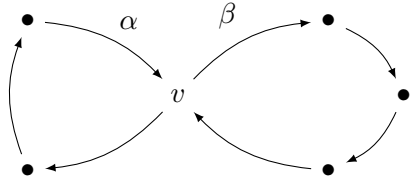
\begin{figure}
\begin{tikzpicture}[>=latex]
\node (-1) at (-2,-1) {$\bullet$};
\node (-2) at (-2,1) {$\bullet$};
\node (1) at (0,0) {$v$};
\node (2) at (2,1) {$\bullet$};
\node (3) at (3,0) {$\bullet$};
\node (4) at (2,-1) {$\bullet$};
\draw[->] (1) to[bend left=20]  (-1);
\draw[->] (-1) to[bend left=20](-2);
\draw[->] (-2) to[bend left=20] node[above right] {$\alpha$}(1);
\draw[->] (1) to[bend left=20]  node [above left] {$\beta$}(2);
\draw[->] (2) to[bend left=20]  (3);
\draw[->] (3) to[bend left=20]  (4);
\draw[->] (4) to[bend left=20]  (1);
\end{tikzpicture}
\caption{The quiver $C_3\star_v C_4$}\label{Fig:CnCm}
\end{figure}
\begin{thm}\label{Thm:PiRelation}
Let $Q$ be a quiver given by the gluing of an $n$-cycle $C(1)$ and an $m$-cycle $C(2)$ at a vertex $v$, for some $n,m\ge1$. Let $\alpha$ be the arrow of $C(1)$ ending in $v$ and let $\beta$ be the arrow of $C(2)$ starting from $v$. The monomial quiver algebra $A=FQ/\langle\alpha\beta\rangle$ is PI.   More precisely, $St_{2k}^{(1)}St_{2k}^{(2)}$ is a polynomial identity for $A$, where $k=\max\{n,m\}$.
\end{thm}
\begin{proof} 
Let us use the shorthand
\[
f(x_1,\dots, x_{4k})\coloneqq St_{2k}(x_1,\cdots, x_{2k})St_{2k}(x_{2k+1},\cdots, x_{4k}).
\]
We need to prove that $f\in\id{A}$. Since $f$ is multilinear  in $4k$ variables, it is enough to evaluate $f$ on $4k$ standard basis elements of $A$. By definition, these are either paths of $C(1)$, or paths of $C(2)$ or of the form $\gamma_2\gamma_1$, where $\gamma_i\in FC(i)$, for $i=1,2$. Let, thus, $\mathcal{S}=(p_1,\cdots, p_{4k})$ be a tuple of $4k$ such elements. If every product of all the elements of $\mathcal{S}$ is zero, then $f(p_1,\dots, p_{4k})$ is zero. Up to a reordering of $\mathcal{S}$, we can hence assume that $p_1p_2\cdots p_{4k}\neq0$.  If $\mathcal{S}\subset FC(2)$ then $St_{2k}(p_1,\cdots, p_{2k})=0$ by Corollary~\ref{Cor:FCn} applied to $C(2)$ and \eqref{Eq:LeibnizStn}. Otherwise, let $\ell$ be the first index such that $t(p_\ell)\in C(1)$. If $\ell\le 2k$ then $p_s\in FC(1)$ for every $s> \ell$, and,  again by Corollary~\ref{Cor:FCn} applied to $C(1)$ and \eqref{Eq:LeibnizStn}, $St_{2k}(p_{2k+1},\cdots, p_{4k})=0$. If $\ell> 2k$ then similarly $St_{2k}(p_{1},\cdots, p_{2k})=0$.  
\end{proof}

\begin{rem}
The quiver algebra considered in Theorem~\ref{Thm:PiRelation} is not finite dimensional.
\end{rem}

\end{document}